\documentclass[reqno]{amsart}
\usepackage{graphicx}
\usepackage{geometry}
\usepackage{amsmath}
\usepackage{amscd}
\usepackage{amsthm}
\usepackage{amsfonts}
\usepackage{amssymb}
\usepackage[colorlinks=true]{hyperref}

\newtheorem{theorem}{Theorem}[section]
\newtheorem{theorem*}{Theorem}

\newtheorem{lemma}[theorem]{Lemma}
\newtheorem{proposition}[theorem]{Proposition}

\theoremstyle{definition}

\newtheorem{example}[theorem]{Example}

\newtheorem{remark}[theorem]{Remark}

\DeclareMathOperator{\End}{End}

\DeclareMathOperator{\Gal}{Gal}
\DeclareMathOperator{\Perm}{Perm}

\DeclareMathOperator{\mult}{mult}

\begin{document}

\newcommand{\gen}[1]{\langle #1 \rangle}
\newcommand{\F}[1]{\mathbb{F}_{p^{#1}}}
\newcommand{\sg}[1]{^{\sigma^{ #1 }}}
\newcommand{\Q}{\mathbb{Q}}

\newcommand{\mmod}{\hbox{\,mod\,}}
\newcommand{\upcarat}{\^\,\hspace{-.5mm}}
\newcommand{\Pf}{{\noindent\it Proof.\ }}
\newcommand{\ePf}{\hfill $\Box$\vspace{1cm}}
\newcommand{\be}{\begin{eqnarray}}
\newcommand{\ee}{\end{eqnarray}}
\newcommand{\ben}{\begin{eqnarray*}}
\newcommand{\een}{\end{eqnarray*}}
\newcommand{\dis}{\displaystyle}
\newcommand{\Cl}{\mathcal C}
\newcommand{\Z}{\mathbb Z}
\newcommand{\Ho}{\hbox{Hom}_{\scriptsize\hbox{R-alg}}}
\newcommand{\C}{\mathbb C}

\title[The Structure of Hopf Algebras Acting on Dihedral Extensions]{The Structure of Hopf Algebras Acting on Dihedral Extensions}

\author{Alan Koch}
\address{Department of Mathematics, Agnes Scott College, 141 E. College Ave., Decatur, GA 30030 USA}
\email{akoch@agnesscott.edu}
\author{Timothy Kohl}
\address{Department of Mathematics and Statistics, Boston University, 111 Cummington Mall, Boston, MA 02215 USA}
\email{tkohl@math.bu.edu} 
\author{Paul J.~Truman}
\address{School of Computing and Mathematics, Keele University, Staffordshire, ST5 5BG, UK}
\email{P.J.Truman@Keele.ac.uk}
\author{Robert Underwood}
\address{Department of Mathematics and Computer Science, Auburn University at Montgomery, Montgomery, AL, 36124 USA}
\email{runderwo@aum.edu}

\date{\today}

\begin{abstract}
We discuss isomorphism questions concerning the Hopf algebras that yield Hopf-Galois structures for a fixed separable field extension $L/K$. We study in detail the case where $L/K$ is Galois with dihedral group $D_p$, $p\ge 3$ prime and give explicit descriptions of the Hopf algebras which act on $L/K$.  We also determine when two such Hopf algebras are isomorphic, either as Hopf algebras or as algebras. For the case $p=3$ and a chosen $L/K$, we give the Wedderburn-Artin decompositions of the Hopf algebras. 
\end{abstract}

\maketitle

{\em keywords:}\ Hopf algebra; Hopf-Galois structure; separable extension; dihedral group


\section{Introduction}

Galois theory for purely inseparable field extensions was first considered by Jacobson \cite{Jacobson44}.  More broadly, Chase and Sweedler defined the notion of a Hopf algebra acting on a purely inseparable extension of fields to obtain a weak analogue to the Fundamental Theorem of Galois theory \cite{ChaseSweedler69}.  The construction of these ``Hopf-Galois'' structures however, applies not only to purely inseparable field extensions but also to separable field extensions, as well as extensions of commutative rings (which will not be considered here).

In the groundbreaking paper \cite{GreitherPareigis87}, Greither and Pareigis obtain a classification of Hopf-Galois structures on separable field extensions $L/K$. The most remarkable aspect of their classification is that it is entirely group-theoretic, depending on $\Gal(E/K)$ and $\Gal(E/L)$, where $E$ is the normal closure of $L/K$: the structure of $L/K$ is irrelevant (aside from these Galois groups). Much of the work to date has focused on counting the number of Hopf-Galois extensions, either directly (see, e.g., \cite{Byott04},\cite{CarnahanChilds99},\cite{Childs07},\cite{FeatherstonhaughCarantiChilds12},\cite{Kohl98},\cite{Kohl07},\cite{Kohl13},\cite{Kohl16}) or through results which facilitate computations (e.g., \cite{Byott96}).

In the three decades since the publication of \cite{GreitherPareigis87}, what has been lacking is a thorough investigation into the structure of the Hopf algebras which produce these Hopf-Galois structures. It is unclear how much the structure of the Hopf algebras also depends on group theory. To this end, we introduce some questions for study, including:
\begin{enumerate}
\item Can a single $K$-Hopf algebra determine more than one Hopf-Galois structure on $L/K$?
\item Can two non-isomorphic $K$-Hopf algebras, each of which giving a Hopf-Galois structure on $L/K$, become isomorphic upon base change to some intermediate field $K\subset F\subset L$? 
\item Can a single $K$-algebra be endowed with multiple coalgebra structures, resulting in multiple (non-isomorphic) Hopf algebras, giving different Hopf-Galois structures on $L/K$? 
\item Can two non-isomorphic $K$-algebras, each of which giving a Hopf-Galois structure on $L/K$ (after being endowed a coalgebra structure), become isomorphic as algebras upon base change to some intermediate field $K\subset F\subset L$? 
\end{enumerate}

It is not known whether the answers to these questions depend on knowledge of the fields, aside from their automorphism groups. Some specific cases--most notably the cases where $\Gal(L/K)$ is cyclic or elementary abelian--have been investigated by the authors in \cite{TARP17c}.  We note that L. Childs \cite[Theorem 5]{Childs13} has shown that abelian fixed-point free endomorphisms
of $\Gal(L/K)$ determine Hopf Galois structures on $L/K$ whose Hopf algebras are isomorphic to $H_\lambda$ (see 
Example \ref{canon} for the definition of $H_\lambda$).  Childs applies his result to the cases where $\Gal(L/K)$ is the symmetric group $S_n$, 
$n\ge 5$, and the dihedral group of order $4n$, $n\ge 2$.  Thus in these cases, Childs has obtained an 
affirmative answer to Question (1).  In this paper, we will focus on the case where $L/K$ is Galois with $\Gal(L/K)=D_p$, the dihedral group of order $2p$ for some prime $p\ge 3$. In this instance, the questions above do in fact have satisfying group theoretic answers. 

We start by reviewing Hopf-Galois theory and the Greither-Pareigis theory that yields the classification of Hopf-Galois structures in the separable case. We then apply Greither-Pareigis theory to describe the Hopf-Galois structures in the case that $L/K$ is Galois with group $D_p$; there are $p+2$ such structures, and while this was known previously (see \cite{Byott04}) we provide a simpler description.   We will show that there are three Hopf algebra isomorphism classes, and that base changing to a proper intermediate field still results in three distinct Hopf algebra classes. On the other hand, there are two $K$-algebra isomorphism classes, and base changing to any intermediate field (even $L$ itself) does not change these isomorphism classes. We then find explicit bases for each of our Hopf algebras, and specializing to the case $p=3$, we give an even more detailed description of the algebra structure. 

In contrast, we point out that the Hopf-Galois theory in the purely inseparable case differs greatly to what is presented here.  For example, if $L=K(x),\ x^p\in K,\;x\not\in K$, is a purely inseparable extension ($p$ prime), then 
$H=K[t]/(t^p)$, $t$ primitive, can act in an infinite number of ways (see, e.g., \cite{Koch15}), allowing for an infinite number of Hopf-Galois structures.

We thank the referee for comments and suggestions which improved the exposition and content of this paper.

\section{Hopf Galois theory}

In this section we recall the notion of a Hopf algebra, a Hopf-Galois extension, and the Greither-Pareigis classification.

A {\em bialgebra} over a field $K$ is a $K$-algebra $B$ together with $K$-algebra maps $\Delta: B\rightarrow B\otimes_K B$ (comultiplication) and $\varepsilon: B\rightarrow K$ (counit) which satisfy the conditions
\[(I\otimes \Delta)\Delta = (\Delta \otimes I)\Delta,\]
\[\mult(I\otimes \varepsilon)\Delta = I = \mult(\varepsilon\otimes I)\Delta,\]
where $\mult: B\otimes_K B\rightarrow B$ is the multiplication map of $B$ and $I$ is the identity map on $B$.
A {\em Hopf algebra} over $K$ is a $K$-bialgebra $H$ with a $K$-linear map $\sigma:H\rightarrow H$ which satisfies
\[\mult(I\otimes \sigma)\Delta(h) = \varepsilon(h)1_H = \mult(\sigma\otimes I)\Delta(h),\]
for all $h\in H$.  A $K$-Hopf algebra $H$ is {\em cocommutative} if $\Delta=\tau\circ \Delta$, where $\tau: H\otimes_K H\rightarrow
H\otimes_K H$, $a\otimes b\mapsto b\otimes a$ is the twist map. 

Let $L$ be a finite extension of $K$ and let ${\mathrm m}: L\otimes_K L\rightarrow L$ denote multiplication in $L$.  Let $H$ be a finite dimensional, cocommutative $K$-Hopf algebra and suppose there is a $K$-linear action of $H$ on $L$ which satisfies
\begin{align*}
h\cdot(xy) &=({\mathrm m}\circ \Delta)(h)(x\otimes y) \\
h\cdot 1 & =\varepsilon(h)1
\end{align*}
for all $h\in H,\;x,y\in L$, and that the $K$-linear map $j: L\otimes_K H \to \End_K(L),\; j(x\otimes h)(y)=x(h\cdot y)$ is an isomorphism of vector spaces over $K$. Then we say $H$ provides a {\it{Hopf-Galois structure}} on $L/K$.

\begin{example}
Suppose $L/K$ is Galois with Galois group $G$. Let $H=K[G]$ be the group algebra, which is a Hopf algebra via $\Delta(g)=g\otimes g,\;\varepsilon(g)=1$, $\sigma(g)=g^{-1}$, for all $g\in G$. The action
\[\Big(\sum r_gg\Big)\cdot x = \sum r_g(g(x)) \]
provides the ``usual'' Hopf-Galois structure on $L/K$ which we call the {\it{classical}} Hopf-Galois structure.
\end{example}

In general, the process of finding a Hopf algebra and constructing an action may seem daunting, but in the separable case Greither and Pareigis \cite{GreitherPareigis87} have provided a complete classification of such structures. Let $L/K$ be separable with normal closure $E$. Let $G=\Gal(E/K)$, $G'=\Gal(E/L)$, and $X=G/G'$. Denote by $\Perm(X)$ the group of permutations of $X$.  A subgroup $N\le \Perm(X)$ is {\em regular} if $|N|=|X|$ and $\eta[xG']\ne xG'$ for all $\eta\ne 1_N, xG'\in X$. 
Let $\lambda: G\rightarrow \Perm(X)$, $\lambda(g)(xG')=gxG'$, denote the left translation map.  A subgroup $N\le \Perm(X)$ is {\em normalized} by $\lambda(G)\le \Perm(X)$ if $\lambda(G)$ is contained in the normalizer of $N$ in $\Perm(X)$. 

\begin{theorem} [Greither-Pareigis]\label{GP} Let $L/K$ be a finite separable extension.  There is a one-to-one correspondence between Hopf Galois structures on $L/K$ and regular subgroups of $\mathrm{Perm}(X)$ that are normalized by $\lambda(G)$.
\end{theorem}

One direction of this correspondence works by Galois descent:\ Let $N$ be a regular subgroup normalized by $\lambda(G)$.  
Then $G$ acts on the group algebra $E[N]$ through the Galois action on $E$ and conjugation by $\lambda(G)$ on $N$, i.e., 
\[g(x\eta) = g(x)(\lambda(g)\eta\lambda(g^{-1})), g\in G, \;x\in E,\; \eta\in N.\]
For simplicity, we will denote the conjugation action of $\lambda(g)\in\lambda(G)$ on $\eta\in N$ by $^g\eta$.
We then define 
\[H=(E[N])^G=\{x\in E[N]:\ g(x)=x, \forall g\in G\}.\]
The action of $H$ on $L/K$ is thus
\[\Big(\sum_{\eta\in N} r_{\eta} \eta \Big)\cdot x = \sum_{\eta\in N}r_{\eta} \eta^{-1}[1_G](x),\] 
see \cite[Proposition 1]{Childs11}.

The fixed ring $H$ is an $n$-dimensional $K$-Hopf algebra, $n=[L:K]$, and $L/K$ has a Hopf Galois structure via $H$ \cite[p. 248, proof of 3.1 (b)$\Longrightarrow$ (a)]{GreitherPareigis87}, \cite[Theorem 6.8, pp. 52-54]{Ch00}.    
By \cite[p. 249, proof of 3.1, (a) $\Longrightarrow$  (b)]{GreitherPareigis87},
\[E\otimes_K H \cong E\otimes_K K[N]\cong E[N],\]
as $E$-Hopf algebras,
that is, $H$ is an {\em $E$-form} of $K[N]$. 

Theorem \ref{GP} can be applied to the case where $L/K$ is Galois with group $G$ (thus, $E=L$, $G'=1_G$, $G/G'=G$).
In this case the Hopf Galois structures on $L/K$ correspond to regular subgroups of $\mathrm{Perm}(G)$ normalized by 
$\lambda(G)$, where $\lambda: G\rightarrow \mathrm{Perm}(G)$, $\lambda(g)(h)=gh$, is the left regular representation. 

\begin{example}\label{class} Suppose $L/K$ is a Galois extension, $G=\Gal(L/K)$. Let $\rho: G\rightarrow \mathrm{Perm}(G)$
be the right regular representation defined as $\rho(g)(h)=hg^{-1}$ for $g,h\in G$.  
Then $\rho(G)$ is a regular subgroup normalized by $\lambda(G)$, since 
$\lambda(g)\rho(h)\lambda(g^{-1}) = \rho(h)$ for all $g,h\in G$; $N$ corresponds to a Hopf-Galois structure with 
$K$-Hopf algebra $H=L[\rho(G)]^G = K[G]$, the usual group ring Hopf algebra with its usual action on $L$.  
Consequently, $\rho(G)$
corresponds to the classical Hopf Galois structure.  
\end{example}

\begin{example}\label{canon} Again, suppose $L/K$ is Galois with group $G$. Let $N=\lambda(G)$.  Then $N$ is a regular subgroup of $\Perm(G)$ which is normalized by $\lambda(G)$, and $N=\rho(G)$ if and only if $N$ abelian.  We denote the corresponding Hopf algebra by $H_{\lambda}$.  If $G$ is non-abelian, then $\lambda(G)$ 
corresponds to the {\em canonical non-classical} Hopf-Galois structure. 
\end{example}

Thus, for $G$ non-abelian there are at least two Hopf-Galois structures on $L/K$.  We remark that if $L/K$ is Galois with $G$ simple and non-abelian, then these are the only Hopf Galois structures on $L/K$ \cite{By04a}.

Note that with the classical and the canonical non-classical structures, the regular subgroup $N\le \Perm(G)$ is isomorphic to $G$. The following example shows that this need not be the case in general.

\begin{example} Let $L=\mathbb{Q}(\sqrt{2},\sqrt{3})$. Then $L/\mathbb{Q}$ is Galois with elementary abelian Galois group $G=\gen{r,s}$ with
\[r(\sqrt{2}+\sqrt{3})=\sqrt{2}-\sqrt{3},\;s(\sqrt{2}+\sqrt{3})=-\sqrt{2}+\sqrt{3}.\]
Let $\eta\in\Perm(G)$ be defined by $\eta(r^is^j)=r^{i-1}s^{i+j-1}$, and let $N=\gen{\eta}$. It is routine to verify that $N$ is a regular subgroup of $\Perm(G)$ which is normalized by $\lambda(G)$. Since $N$ is cyclic of order $4$, $N\not\cong G$.
\end{example}

\section{The group $D_p$}

Throughout this section, we let $D_p$ denote the dihedral group of order $2p$ for $p$ an odd prime.  Explicitly, we write
\[D_p=\gen{r,s : r^p=s^2=rsrs=1}.\] 
Let $L/K$ be a Galois extension with group $D_p$.  We shall describe all of the regular subgroups of $\Perm(D_p)$ normalized by 
$\lambda(D_p)$, and then address the isomorphism questions given in the Introduction.   By Examples \ref{class} and \ref{canon}  we have regular subgroups $\rho(D_p),\;\lambda(D_p)$ normalized by $\lambda(D_p)$.   
We construct other regular subgroups of $\Perm(D_p)$.

\begin{lemma}
Pick $0\le c \le p-1$. Let $\eta_c=\lambda(r)\rho(r^cs)\in\Perm(D_p)$, and let $N_c=\gen{\eta_c}$. Then $N_c\cong C_{2p}$, the cyclic group of order $2p$, and the $N_c$ are distinct as sets.   Moreover, $N_c$ is a regular subgroup of $\Perm(D_p)$ normalized by $\lambda(D_p)$.  
\end{lemma}

\begin{proof} Suppose $0\le c\le p-1$.  Since left and right representations commute with each other, $\eta_c^i=\lambda(r^i)\rho((r^cs)^i)$. As $|r|=p$ and $|r^cs|=2$ it follows that $|N_c|=2p$, thus $N_c\cong C_{2p}$.  Now suppose $0\le d\le p-1$,
$c\not = d$.  Since $N_c\cong C_{2p}$, it contains a unique element of order $2$, which is $\eta_c^p=\rho(r^cs)$.  Similarly, the unique element of order $2$ in $N_d$ is  $\eta_d^p=\rho(r^ds)$.  Since $c\not = d$, this shows that $N_c\not = N_d$. 

It remains to show that the stabilizer in $N_c$ of any element in $D_p$ is trivial, and that $N_c$ is normalized by $\lambda(D_p)$. 
For the remainder of the proof, we write $\eta$ for $\eta_c$ and $N$ for $N_c$.    Let $x\in D_p$ and suppose $\eta^i[x]=x$. Then
\[x=\eta^i[x]= \lambda(r^i)\rho((r^cs)^i)[x] = r^ix(r^cs)^{-i},\]
and so,
\[1_{D_p} = x^{-1}r^ix(r^cs)^{-i}=r^{\pm i}(r^cs)^{-i}.\]
which cannot happen unless $i=0$.  Hence $\eta^i = 1_{N}$ and $N\le \Perm(D_p)$ is regular.

We now show that $N$ is normalized by $\lambda(D_p)$.  Of course, it suffices to show that $^r\eta\in N$ and $^s\eta\in N$. We have, for $x\in D_p$,
\begin{align*}
^r\eta[x]&=\lambda(r)\big(\lambda(r)\rho(r^cs)\big)\lambda(r^{-1})[x] = rxr^cs = \eta[x]\\
^s\eta[x]&=\lambda(s)\big(\lambda(r)\rho(r^cs)\big)\lambda(s)[x]=srsxr^cs = r^{-1}xr^cs =r^{-1}x(r^cs)^{-1}= \eta^{-1}[x].
\end{align*} 
Thus $^r\eta=\eta,\;^s\eta=\eta^{-1}$, and $N$ is normalized by $\lambda(D_p)$.
\end{proof}

By \cite[Corollary 6.5]{Byott04}, the collection $\{\rho(D_p),\lambda(D_p),N_0,\dots,N_{p-1}\}$ is the complete set of all regular subgroups of $\Perm(D_p)$ normalized by $\lambda(D_p)$, hence the corresponding Hopf algebras give all of the Hopf-Galois structures on $L/K$.   We will denote the Hopf algebra corresponding to $N_c$ by $H_c$ for all $c$.

\section{The Hopf algebra isomorphism questions}

Let $L/K$ be Galois with group $D_p$.   We have seen that the Hopf algebras which give Hopf-Galois structures on $L/K$ are
\[\{K[D_p],H_{\lambda},H_0,\dots, H_{p-1}.\} \]
Here, we will investigate when two of these Hopf algebras are isomorphic. Note that throughout this section, when working with Hopf algebras, ``isomorphic'' refers to isomorphic as Hopf algebras; considering isomorphisms as algebras will be discussed in the next section.

Clearly, $H_c$ cannot be isomorphic to either $K[D_p]$ or $H_{\lambda}$ since it is commutative. It remains to determine whether $K[D_p]\cong H_{\lambda}$ or whether $H_c\cong H_d$ for some $0 \le c < d\le p-1$.

Generally, Hopf isomorphism questions reduce to group isomorphism questions.

\begin{proposition}\label{me} Let $L/K$ be a finite separable extension, with Galois closure $E$, and let 
$G=\Gal(E/K)$, $G'=\Gal(E/L)$.   Let $X=G/G'$. Let $N,N'$ be regular subgroups of $\Perm(X)$ which are normalized by $\lambda(G)$, and let $H,H'$ be their corresponding Hopf algebras. If $H\cong H'$ then $N\cong N'$.
\end{proposition} 

\begin{proof} If $H\cong H'$ then $(E[N])^G\cong (E[N'])^G$, hence $E\otimes_K (E[N])^G\cong 
E\otimes_K (E[N'])^G$. However $E\otimes_K (E[N])^G\cong E[N]$ and similarly for $N'$, hence 
$E[N]\cong E[N']$. Since group algebras (over the same field) are isomorphic as Hopf algebras 
if and only if their groups are isomorphic, the result follows.
\end{proof}

The converse to Proposition \ref{me} is not true as we shall see below in Proposition \ref{rho-lam-not-hopf-iso}.  However, we have:

\begin{proposition}\label{PTthing} Using the notation as in Proposition \ref{me}, $H\cong H'$ if and only if there exists an isomorphism $\phi:N\to N'$ which respects the actions of $G$.
\end{proposition}

\begin{proof}   See \cite[Corollary 2.3]{TARP17c}. 
\end{proof}

We can use this proposition to show that the $H_c$ are all isomorphic.

\begin{proposition} \label{H_c-iso}  For $0\le c,d\le p-1$ we have $H_c\cong H_d$.
\end{proposition}

\begin{proof} Define $\phi:N_c \to N_d$ by $\phi(\eta_c)=\eta_d$. This is clearly an isomorphism; it remains to show that it respects the $D_p$-actions. But since the $D_p$-actions are identical with respect to the generators of the groups this is immediate:
\begin{align*}
\phi(\;^{r}\eta_c)&=\phi(\eta_c)=\eta_d = \;^r\eta_d = \;^r\phi(\eta_c)\\
\phi(\;^{s}\eta_c)&=\phi(\eta_c^{-1})=\eta_d^{-1} =\; ^s\eta_d = \;^s\phi(\eta_c).
\end{align*}
\end{proof}

On the other hand, we have:

\begin{proposition} \label{rho-lam-not-hopf-iso}  $K[D_p]\not\cong H_{\lambda}$.
\end{proposition}

\begin{proof} Suppose $K[D_p]\cong H_{\lambda}$. Then there exists an isomorphism $\phi:\rho(D_p) \to \lambda(D_p)$ which respects the $D_p$-actions. Note that $D_p$ acts trivially on $\rho(D_p)$. Pick $1\le i \le p-1$ such that $\phi(\rho(r))=\lambda(r^i)$. Then 
\[\phi(\;^s\rho(r)) = \phi(\rho(r)) =\lambda(r^i) \]
while
\[ ^s\phi(\rho(r))=\;^s(\lambda(r^i))=\lambda(s)\lambda(r^i)\lambda(s)=\lambda(sr^is)=\lambda(r^{-i}),\]
which is a contradiction since $r^i=r^{-i}$ if and only if $i=0$. Thus, $K[D_p]\not\cong H_{\lambda}$.
\end{proof}

\begin{remark} One could also prove Proposition \ref{rho-lam-not-hopf-iso} as follows.  Over $L$, $K[D_p]$
and $H_\lambda$ are isomorphic to $L[D_p]$ as Hopf algebras, thus their duals $K[D_p]^*$ and $H_\lambda^*$
are finite dimensional as algebras over $K$ and separable (as defined in \cite[6.4, page 47]{Waterhouse79}). 
Using the classification of such $K$-algebras \cite[6.4, Theorem]{Waterhouse79}, we conclude that $K[D_p]^*$ and $H_\lambda^*$ are not isomorphic 
as $K$-Hopf algebras, and so neither are $K[D_p]$ and $H_\lambda$.  In fact, by \cite[6.3, Theorem]{Waterhouse79},
$K[D_p]^*$ and $H_\lambda^*$ are not isomorphic as $K$-algebras, and consequently, $K[D_p]$ and $H_\lambda$ are not isomorphic as $K$-coalgebras.  As we will show in Section 5, however, $K[D_p]\cong H_\lambda$ as 
$K$-algebras.
\end{remark}

Picking $c=0$, we obtain the following.

\begin{theorem} There are three $K$-Hopf algebras which provide Hopf-Galois structures on a dihedral extension $L/K$ of degree $2p$, namely:
\begin{enumerate}
\item The group algebra $K[D_p]$, which provides the classical structure
\item The Hopf algebra $H_{\lambda}$, which provides the canonical non-classical structure
\item A commutative $K$-Hopf algebra $H_0$ which provides $p$ different structures.
\end{enumerate}
\end{theorem}

We now wish to consider whether any two of the Hopf algebras above become isomorphic after base change to an intermediate field $K\subset F \subset L$. This question is relatively easy to answer in the dihedral case. Since a $K$-Hopf algebra $H$  is commutative if and only if $F\otimes_K H$ is commutative the Hopf algebra $H_0$ is not isomorphic to either $K[D_p]$ or $H_{\lambda}$ after base change. What remains is to determine whether we can have $F\otimes_K K[D_p] \cong F\otimes_K H_{\lambda}$.
But this is also easy:\ Since the center $Z(D_p)$ is trivial, the subgroup of $D_p$ defined as $\{g\in D_p:\ ^g\lambda(h)=\lambda(h), \forall h\in D_p\}$ is trivial.   Hence by \cite[Corollary 3.2]{GreitherPareigis87}, $F=L$ is the smallest field extension of $L$ for which $F\otimes_K H_\lambda\cong F[D_p]$.  Thus $F=L$ is minimal so that $F\otimes_K K[D_p]\cong F\otimes_K H_\lambda$.

\section{The algebra structure}

Let $L/K$ be Galois with group $D_p$, $\Q\subseteq K$.   In this section we investigate the question of when two Hopf algebras providing Hopf-Galois structures on $L/K$ are isomorphic as algebras. Throughout this section, when working with Hopf algebras, ``isomorphic'' refers to isomorphic as algebras.  Since $\mathrm{char}(K)$ does not divide $[L:K]$, Maschke's theorem and a result of Amitsur \cite[Theorem 1]{Am57} shows that the Hopf algebras are left semisimple.

Of course, $H_c\cong H_d$ for all $0\le c,d\le p-1$, and these (Hopf) algebras are not isomorphic to either $K[D_p]$ or $H_{\lambda}$ since they are commutative; they remain non-isomorphic after base change for the same reason.

It remains to consider the classical and the canonical non-classical structures, $K[D_p]$ and $H_\lambda$, respectively.   
Our main tool will be the Wedderburn-Artin decomposition.

The decomposition of $\Q[D_p]$ is given in \cite[Example (7.39)]{CurtisReiner81}. 

\begin{theorem} \label{group-algebra}  Let $D_p$, $p\ge 3$ prime, be the dihedral group of order $2p$.
Then
\[\Q[D_p]\cong \Q\times \Q\times \mathrm{Mat}_2(\Q(\zeta_p+\zeta_p^{-1})),\]
where $\zeta_p$ denotes a primitive $p$th root of unity.  
\end{theorem}

Let $K$ be a field extension of $\Q$ with $E=K\cap \Q(\zeta_p+\zeta_p^{-1})$ and $l=[\Q(\zeta_p+\zeta_p^{-1}): E]$.
Then
\be \label {decomp-p}
K[D_{p}]\cong \left\{\begin{array}{l l} K \times K \times \Big(\mathrm{Mat}_2\big(K(\zeta_p+\zeta_p^{-1})\big)\Big )^{(p-1)/(2l)} & \zeta_p+\zeta_p^{-1}\not\in K \\
 K \times K \times \Big(\mathrm{Mat}_2(K)\Big)^{(p-1)/2} & \zeta_p+\zeta_p^{-1}\in K. \end{array}\right. 
\ee
The former case follows by noting that $E=\Q(\alpha)$ for $\alpha\in E$ satisfying an irreducible monic polynomial
of degree $(p-1)/(2l)$ over $\Q$.   The latter case follows from observing that if $\zeta_p+\zeta_p^{-1}\in K$, then the $(p-1)/2$ $2$-dimensional 
irreducible representations of $D_p$ over $\C$
correspond to $(p-1)/2$ characters with values in $K$ \cite[Chapter 5, 5.3]{Se77}.

What can be said about the decomposition of $H_{\lambda}$?   Since $H_{\lambda}$ is left semisimple, the $K$-algebra $H_{\lambda}$ decomposes into a product of matrix rings over division rings,
\[H_{\lambda}\cong \mathrm{Mat}_{q_1}(R_1)\times  \mathrm{Mat}_{q_2}(R_2)\times\cdots\times \mathrm{Mat}_{q_t}(R_t).\]
The division rings $R_i$ are finite dimensional $K$-algebras.   

Now, there are exactly two $1$-dimensional irreducible representations of $D_p$, with characters $\chi_1$ and $\chi_2$, corresponding to mutually orthogonal idempotents 
\[e_1={1\over 2p}\sum_{g\in D_p}\chi_1(g^{-1})g,\quad e_2={1\over 2p}\sum_{g\in D_p}\chi_2(g^{-1})g\] 
in $L[D_p]$.  Both $e_1$ and $e_2$ are fixed by the action of $D_p$, hence $e_1,e_2\in H_{\lambda}$. It follows that
\be \label {H-lambda}
H_{\lambda}\cong K\times K\times \prod_{j=1}^m \mathrm{Mat}_{q_j}(R_j), 
\ee
where $q_j\ge 1$ are integers and $R_j$ are division rings.   For later use, we set $S_j=\mathrm{Mat}_{q_j}(R_j)$ for $1\le j\le m$.   
Observe that  
\[\sum_{j=1}^m \dim_K(R_j)\cdot q_j^2 = 2(p-1).\]

For the moment we assume $p=3$, so that $\sum_{j=1}^m \dim_K(R_j)\cdot q_j^2=4$.   
Since the dimension of a division algebra over its center is a perfect square, we conclude that
\be \label {decomp-3}
H_{\lambda}\cong K\times K\times \mathrm{Mat}_{q}(R),
\ee  
where $1\le q\le 2$ and $R$ is a division ring.   If $q=1$, then the corresponding division ring $R$ is non-commutative.  If $q=2$, then $R=K$.  

Assume that the base field $K=\Q$, and let $L$ be the splitting field of $x^3-v$ over ${\mathbb Q}$ where $v$ is not a $3$rd power in $\Q$.  Let $\zeta_3$ denote a primitive $3$rd root of unity and let 
$\alpha  = \root 3\of v$.   Then $L={\mathbb Q}(\alpha,\zeta_3)$ is Galois with group $D_3$.  
The Galois action is given as $r(\alpha)=\zeta_3\alpha$, $r(\zeta_3)=\zeta_3$, $s(\alpha)=\alpha$, 
$s(\zeta_3)=\zeta_3^2$.   Let
\[b=\alpha s+\alpha\zeta_3 sr+\alpha\zeta_3^2 sr^2.\] 
Then $b\in L[D_3]$, and since $b$ is fixed by all elements of $D_3$, $b\in H_{\lambda}$.  Moreover, direct computation yields
$b^2=0$.  Thus the only possibility is that $q=2$ in (\ref{decomp-3}), and so,
\[H_{\lambda}\cong \Q\times \Q\times \mathrm{Mat}_2(\Q).\]

Since $\Q[D_3]\cong \Q\times \Q\times \mathrm{Mat}_2(\Q)$ by Theorem \ref{group-algebra}, here we have an instance of $H_{\lambda}\cong \Q[D_3]$ as $\Q$-algebras.  
Surprisingly, this is true for {\it any} $L/K$ Galois with group $D_p$, $p\ge 3$.  In fact,
it holds even more generally:

\begin{theorem} [Greither]\label{alg-iso2}  Let $\Q\subseteq K$, and let $L/K$ be a Galois extension with group $G$. 
Then 
\[H_{\lambda}\cong K[G]\]
as $K$-algebras.
\end{theorem}

\begin{proof}  We prove the special case where $G=D_p$.  We thank C. Greither for the method of proof.   

From (\ref{decomp-p}) we obtain
\[L\otimes_K H_{\lambda}\cong L\otimes_K K[D_p]\cong L[D_{p}]\cong \left\{\begin{array}{l l} L \times L \times 
\Big(\mathrm{Mat}_2\big(L\otimes_K K(\zeta_p+\zeta_p^{-1})\big)\Big )^{(p-1)/(2l)} & \zeta_p+\zeta_p^{-1}\not\in K \\
 L \times L \times \Big(\mathrm{Mat}_2(L)\Big)^{(p-1)/2} & \zeta_p+\zeta_p^{-1}\in K, \end{array}\right.\]
with $l=[\Q(\zeta_p+\zeta_p^{-1}): E]=[K(\zeta_p+\zeta_p^{-1}): K]$, 
$E=K\cap \Q(\zeta_p+\zeta_p^{-1})$.  Thus the decomposition of $L[D_p]$ contains components of the form $L$ and 
$\mathrm{Mat}_2\big(L\otimes_K F\big )$, with $F=K(\zeta_p+\zeta_p^{-1})$.

The Hopf algebra $H_{\lambda}$ descends from $L[D_p]$ via the action of $D_p$ as usual:\ an element of $D_p$ acts by conjugation on $D_p$ and by the Galois action on $L$.   A character of $D_p$ has the same value on conjugate elements in $D_p$
 \cite[Chapter 2, 2.1, Proposition 1(iii)]{Se77}.  Thus the central indecomposable idempotents of $\C[D_p]$ as constructed in 
 \cite[Chapter 6, 6.3, Exercise 6.4]{Se77} are fixed by conjugation by elements of $D_p$.   Let
$e$ be a central indecomposable idempotent in $K[D_p]$.   Since $e$ is in the center of $\C[D_p]$, $e$ is a sum of 
central indecomposable idempotents of $\C[D_p]$ \cite[Chapter 6, 6.3, Exercise 6.4]{Se77}.  Thus, conjugation by elements of $D_p$ fixes $e$.  

Now, the central indecomposable idempotents of $K[D_p]$ correspond to the components in the decomposition 
(\ref{decomp-p}); let $M$ be the component
of $K[D_p]$ corresponding to $e$. For $\alpha\in M$, $g\in D_p$, $g\alpha g^{-1}=g\alpha e g^{-1} = g\alpha g^{-1}
geg^{-1} = g\alpha g^{-1}e\in M$, and so, congugation by $g$, which is an automorphism of $K[D_p]$, restricts to an automorphism of each component of $K[D_p]$.   Hence the action of $D_p$ preserves the Wedderburn-Artin components in the decomposition of $L\otimes_K K[D_p]\cong L[D_p]$.  So $D_p$ can be thought of as acting on these components.  The two copies of $L$ in the decomposition of $L[D_p]$ descend to the two copies of $K$ in the decomposition (\ref{H-lambda})
of $H_{\lambda}$, and each copy of $\mathrm{Mat}_2\big(L\otimes_K F\big )$ descends to a component 
$S=S_j$ in the decomposition (\ref{H-lambda}) of $H_{\lambda}$; the $K$-algebra $S$ is an $L$-form of $\mathrm{Mat}_2\big(F\big )$.   We want to show that $S\cong \mathrm{Mat}_2\big(F\big )$ as $K$-algebras, and so $H_{\lambda}\cong K[D_p]$.  
 
Let $\mathrm{Aut}(\mathrm{Mat}_2(F))$ denote the automorphism group scheme of 
$\mathrm{Mat}_2(F)$ in the sense of \cite[\S 7.6]{Waterhouse79}.   By \cite[Theorem, p. 137]{Waterhouse79} the isomorphism classes of $L$-forms of $\mathrm{Mat}_2(F)$ correspond to the cohomology set $\mathrm{H}^1(D_p,\mathrm{Aut}(\mathrm{Mat}_2(L\otimes_K F))$.  
A $1$-cocycle (crossed homomorphism)
is a function $f: D_p\rightarrow \mathrm{Aut}(\mathrm{Mat}_2(L\otimes_K F))$ which satisfies $f(gh)=f(g)\circ (g\cdot f(h))$, 
for $g,h\in D_p$.  The action $g\cdot f(h)$ of the element $g\in D_p$ on the automorphism $f(h)$ in  
$\mathrm{Aut}(\mathrm{Mat}_2(L\otimes_K F))$ is induced by the Galois action on $L$:\  we have 
\[g\cdot f(h)=(g\otimes I_F)f(h)(g^{-1}\otimes I_F),\]
where $g,h\in D_p$, and $I_F$ is the identity map on $F$.   The trivial element in 
$\mathrm{H}^1(D_p,\mathrm{Aut}(\mathrm{Mat}_2(L\otimes_K F))$ is represented by the $1$-cocycle

\[g\mapsto \phi(g\otimes I_F)\phi^{-1}(g^{-1}\otimes I_F),\]

where $\phi$ is any element of $\mathrm{Aut}(\mathrm{Mat}_2(L\otimes_K F))$.   The $L$-form $S$ comes from a particularly simple $1$-cocycle 
$\hat f: D_p\rightarrow \mathrm{Aut}(\mathrm{Mat}_2(L\otimes_K F))$.  For $g\in D_p$, $\hat f(g)$ is conjugation by 
$g$ on $\mathrm{Mat}_2(F)\subseteq \mathrm{Mat}_2(L\otimes_K F )$.

Let $\mathrm{Inn}(\mathrm{Mat}_2(L\otimes_K F))$ denote the group of inner automorphisms.  Since every element of 
$\mathrm{Inn}(\mathrm{Mat}_2(L\otimes_K F))$ is given as conjugation by some element 
of $\mathrm{GL}_2(L\otimes_K F)$, there is a surjection
of groups $\psi: \mathrm{GL}_2(L\otimes_K F)\rightarrow \mathrm{Inn}(\mathrm{Mat}_2(L\otimes_K F))$ with 
$\ker(\psi)=(L\otimes_K F)^\times$.  Thus, there is an induced map in cohomology  
\[\mathrm{H}^1(D_p,\mathrm{GL}_2(L\otimes_K F))\stackrel{\psi}{\longrightarrow }
\mathrm{H}^1(D_p,\mathrm{Inn}(\mathrm{Mat}_2(L\otimes_K F))).\]

A $1$-cocycle $q\in \mathrm{H}^1(D_p,\mathrm{GL}_2(L\otimes_K F))$ is a function 
$q: D_p\rightarrow \mathrm{GL}_2(L\otimes_K F)$ which satisfies $q(gh)=q(g)(g\cdot q(h))$.  The action of $g\in D_p$ on $q(h)\in \mathrm{GL}_2(L\otimes_K F)$ is through the Galois action on $L$ as above.

There is a special $1$-cocycle in $\mathrm{H}^1(D_p,\mathrm{GL}_2(L\otimes_K F))$:\ the function $\hat q$
in which each $g\in D_p$ is identified with its image in $\mathrm{Mat}_2(F)\subseteq \mathrm{Mat}_2(L\otimes_K F)$
under the map $K[D_p]\rightarrow \mathrm{Mat}_2(F)$.  So conjugation by $g$ in $\mathrm{Mat}_2(F)$ is precisely the conjugation action of 
$g$ on $K[D_p]$ restricted to the component $\mathrm{Mat}_2(F)$.  It follows that
\[\psi([\hat q])=[\hat f],\]
where $[\hat q]$ denotes the equivalence class of $\hat q$ in $\mathrm{H}^1(D_p,\mathrm{GL}_2(L\otimes_K F))$
and $[\hat f]$ denotes the equivalence class of $\hat f$ in $\mathrm{H}^1(D_p,\mathrm{Aut}(\mathrm{Mat}_2(L\otimes_K F))$.
The class $[\hat f]$ corresponds to the isomorphism class of the $L$-form $S$.    

Now if $\zeta_p+\zeta_p^{-1}\in K$, then $F=K$. Thus, $L=L\otimes_K F$ and
\[\mathrm{H}^1(D_p,\mathrm{GL}_2(L\otimes_K F))=\mathrm{H}^1(D_p,\mathrm{GL}_2(L)).\]
By Hilbert's Theorem 90, $\mathrm{H}^1(D_p,\mathrm{GL}_2(L))$ is trivial, and so, $[\hat q]$ is trivial, and consequently, 
$[\hat f]$ is trivial.  It follows that
$S\cong \mathrm{Mat}_2(K)$ as $K$-algebras, and so $H_\lambda\cong K[D_p]$. 

If $\zeta_p+\zeta_p^{-1}\not\in K$, then C. Greither has provided a generalization of Hilbert's Theorem 90 to yield 
$\mathrm{H}^1(D_p,\mathrm{GL}_2(L\otimes_K F)$ trivial.  As above, $[\hat f]$ is trivial, and so $S\cong \mathrm{Mat}_2(F)$ as $K$-algebras.  It follows that 
$H_\lambda\cong K[D_p]$. 
\end{proof}

We summarize our findings in this section.  The Hopf algebras that provide Hopf-Galois structures in the case that $L/K$ is Galois with group $D_p$
fall into two $K$-algebra isomorphism classes represented by $K[D_p]$ and $H_0$.   So a single $K$-algebra (e.g., $K[D_p]$)
can be be endowed with multiple coalgebra structures, resulting in multiple (non-isomorphic) Hopf algebras 
(e.g., $K[D_p]$, $H_\lambda$) giving different Hopf-Galois structures on $L/K$ (e.g., classical and canonical non-classical). 

\section{Explicit structure computations}

Let $L/K$ be Galois with group $D_p$, $\Q\subseteq K$.  We find generators over $K$ for the Hopf algebras 
$K[D_p]$, $H_{\lambda}$, $H_0, H_1, \dots, H_{p-1}$ constructed above. Let $L^{\gen{r}}$ be the unique quadratic extension of $K$ contained in $L$. Pick $d\in L$ such that $L^{\gen{r}}=K(\sqrt{d})$. Note that $s(\sqrt{d})=-\sqrt{d}$. Additionally, let $y\in L$ be so that $K(y)=L^{\gen{s}}$.

The simplest case, of course, is $K[D_p]$: it has a $K$-basis $D_p$, and $\{r,s\}$ generates $K[D_p]$  as a $K$-algebra.

We next turn to $H_{\lambda}$. Suppose $h\in H_{\lambda}$. Identifying $D_p$ with $\lambda(D_p)$, we have $h\in L[D_p]$ 
with $h$ fixed by $D_p$. Let 
\[h=\sum_{i=0}^{p-1} a_i r^i + \sum_{i=0}^{p-1}b_i r^is,\; a_i,b_i \in L.\]
Then
\[^rh=\sum_{i=0}^{p-1} r(a_i) r^i + \sum_{i=0}^{p-1} r(b_i) r^{i+1}sr^{-1} = \sum_{i=0}^{p-1} r(a_i) r^i + \sum_{i=0}^{p-1}r(b_i) r^{i+2}s,\]
and since $^rh=h$ we have $a_i\in L^{\gen{r}}$ and $r(b_i)=b_{i+2}$ for all $i$ (where $i+2$ is considered $\mmod p$). Thus, $b_i = r^{i(p+1)/2}(b_0)$ for all $i$.

Furthermore,
\[^sh=\sum_{i=0}^{p-1} s(a_i) sr^is + \sum_{i=0}^{p-1} s(b_i)s r^{i}s^2=\sum_{i=0}^{p-1} s(a_i) r^{-i} + \sum_{i=0}^{p-1} s(b_i) r^{-i}s,\]
which after interchanging $i$ with ${p-i}$ for all $i\ne 0$ gives
\[^sh =s(a_0)+s(b_0)s+ \sum_{i=1}^{p-1} s(a_{p-i}) r^{i} + \sum_{i=1}^{p-1} s(b_{p-i}) r^{i}s\]
and so $a_0\in K,\; s(a_i)=a_{p-i},s(b_0)=b_0$. Note that $s(b_i)=b_{p-i}$ as well, but this followed previously since 
\[s(b_i)=s r^{i(p+1)/2}(b_0) =  r^{-i(p+1)/2}s(b_0) =  r^{-i(p+1)/2}(b_0)= r^{(p-i)(p+1)/2}(b_0)=b_{p-i}. \]
Since $r^{(p-i)(p+1)/2}=r^{i(p-1)/2}$ it follows that
\[H_{\lambda}=\left\{a_0+\sum_{i=1}^{(p-1)/2} (a_i r^i + s(a_i) r^{-i}) + b_0s + 
\sum_{i=1}^{p-1}r^{i(p-1)/2}(b_0)r^{-i}s:\ a_0\in K, a_i\in L^{\gen{r}}, b_0\in L^{\gen{s}}\right\}.\]

\begin{example} \label{H-3} Suppose $p=3$. Then 
\[H_{\lambda}=\left\{a_0+a_1r + s(a_1)r^{2} + b_0s+ r(b_0)sr+r^2(b_0)sr^2:\ a_0\in K, a_1\in L^{\gen{r}}, b_0\in L^{\gen{s}}\right\}.\]
A $K$-basis for $H_\lambda$ is
\[ \{1,r+r^2,\sqrt{d}(r-r^2), s+rs+r^2s,ys+r(y)rs+r^2(y)r^2s,y^2s+r(y^2)rs+r^2(y^2)r^2s\}.\]
\end{example}

We next consider the structure of $H_c$, $0\le c\le p-1$. If we write
\[h=\sum_{i=0}^{2p-1} a_i \eta^i,\ \ \eta=\eta_c,\]
where $a_i \in L$ for all $i$ then,  since $r$ acts trivially on $\eta$,
\[^rh= \sum_{i=0}^{2p-1} r(a_i) \eta^i =\sum_{i=0}^{2p-1} a_i \eta^i. \]
Thus $r(a_i)=a_i$ for all $i$, hence $a_i \in L^{\gen{r}}$. Also, since $^s\eta=\eta^{-1}$,
\[^sh= \sum_{i=0}^{2p-1} s(a_i) \eta^{-i} =\sum_{i=0}^{2p-1} a_i \eta^i, \]
from which it follows that $s(a_0)=a_0$ and $s(a_i)=a_{2p-i}$ for all $i>0$. In particular, $s(a_p)=a_p$, so $a_0,a_p\in K$. Thus,
\[H_c = \left\{a_0+ a_p \eta^p+\sum_{i=1}^{p-1}\big(a_i\eta^i+s(a_i)\eta^{-i}\big):\  a_0, a_p\in K, a_i\in L^{\gen{r}}, 1\le i\le p-1\right\}, \]
and $H_{c}$ has $K$-basis
\[\{1,\eta^p,\eta+\eta^{-1},\eta^2+\eta^{-2},\dots, \eta^{p-1}+\eta^{-(p-1)}, \sqrt{d}(\eta-\eta^{-1}),\sqrt{d}(\eta^2-\eta^{-2}),\dots, \sqrt{d}(\eta^{p-1}-\eta^{-(p-1)} ) \}.\]

\begin{example} \label{Hc-3}   Suppose $p=3$.  Then
\[H_c = \left\{a_0+ a_3\eta^3+a_1\eta+s(a_1)\eta^5+a_2\eta^2+s(a_2)\eta^4:\ a_0, a_3\in K, a_1, a_2\in L^{\gen{r}}\right\}, \]
and $H_{c}$ has $K$-basis
\[\{1,\eta^3,\eta+\eta^{5},\eta^2+\eta^{4},\sqrt{d}(\eta-\eta^{5}),\sqrt{d}(\eta^2-\eta^{4})\}.\]
\end{example}

\section{Example:\ Hopf Galois structures in the case $D_3$}

We close with an analysis of the Hopf Galois structures in the case $p=3$.  Let $L/K$ be any Galois extension with group $D_3$, 
$\Q\subseteq K$.  As we have seen, there are two regular subgroups normalized by $\lambda(D_3)$ and isomorphic to $D_3$, namely, $\rho(D_3)$ and $\lambda(D_3)$, and three regular subgroups normalized by $\lambda(D_3)$ and isomorphic to $C_6$, the cyclic group of order $6$, namely, $N_0$, $N_1$ and $N_2$.

By Proposition \ref{rho-lam-not-hopf-iso} $K[D_3]\not\cong H_{\lambda}$, as $K$-Hopf algebras, and by
Theorem \ref{alg-iso2}
$K[D_3]\cong H_{\lambda}$, as $K$-algebras, with Wedderburn-Artin decomposition
\[K[D_3]\cong H_{\lambda}\cong K\times K\times \mathrm{Mat}_2(K).\]
By Proposition \ref{H_c-iso}, $H_0\cong H_1\cong H_2$ as Hopf algebras, and hence as $K$-algebras.   

We seek the Wedderburn-Artin decomposition and the Hopf algebra structure of $H_0$ (hence of $H_1$ and $H_2$).  
In contrast to the situation with $H_{\lambda}$, the structure of $H_0$
seems to depend on the extension $L/K$, specifically on the fixed field $L^{\gen{r}}$.   

Here is how we can compute the structure of $H_0$.  By \cite[Corollary 3.6]{Kohl13}, 
\[\{g\in \lambda(D_3):\ ^g\eta=\eta, \forall \eta\in N_0\}\]
is precisely the $3$-Sylow subgroup $\lambda(\langle r\rangle)\le \lambda(D_3)$, which we identify with $\langle r\rangle$.  
There is an induced action of $D_3/\langle r\rangle$ on $L[N_0]$.   Note that $D_3/\langle r\rangle\cong C_2$, the cyclic group of order $2$.   By the Fundamental Theorem of Galois theory, 
$D_3/\langle r\rangle\cong C_2$ is the group of the Galois extension $F/K$, where $F=L^{\gen{r}}$;
$F$ is a quadratic extension of $K$.  We write $F=K[z]/(z^2-b)$ for $b\in K$, $z$ indeterminate.    

Now, $D_3/\langle r\rangle\cong C_2$ can be viewed as the 
group of automorphisms of $N_0\cong C_6$.    We have
\[H_0=(L[N_0])^{D_3}=(F[C_6])^{C_2},\]
where the action of $C_2$ on $F[C_6]$ is  by the Galois group on $F$ and as automorphisms on $C_6$. 

Now, $F$ is a $C_2$-Galois extension of $K$, \cite[page 130]{HP86}.  So by \cite[Theorem 5]{HP86}, $F$ corresponds to an $F$-Hopf algebra form of $K[C_6]$, namely, $(F[C_6])^{C_2}$, which must of course be $H_0$.   

And so, $H_0$ is the fixed ring of $F[C_6]$ under the action of $C_2$, and $H_0$ is an $F$-form of $K[C_6]$.  
Under these conditions, $H_0$ can be characterized.  Let $x,y$ be indeterminates and recall
$F=K[z]/(z^2-b)$.   The method of Haggenm\"{u}ller and Pareigis in \cite[Theorem 6, p. 134]{HP86} applies to 
yield $H_0\cong K[x,y]/I$, 
\[I=(y^2-bx^2+u,(x-2)(x-1)(x+1)(x+2),(x-1)(x+1)(xy)),\]
with $u\in K^\times$, $u=4b$.  The Hopf algebra structure of $H_0$ is defined by
\[\Delta(\bar x)={1\over 2}\bar x\otimes \bar x+{1\over 2b}\bar y\otimes \bar y,\]
\[\Delta(\bar y) =  {1\over 2}\bar x\otimes \bar y+{1\over 2}\bar y\otimes \bar x,\]
\[\varepsilon(\bar x)=2,\quad \varepsilon(\bar y)=0,\quad \sigma(\bar x)=\bar x,\quad \sigma(\bar y)=-\bar y.\]
where $\bar x=x\mmod I$, $\bar y= y\mmod I$. 

\begin{remark}  The Hopf algebra structure of $H_0$ does not depend on the choice of generator for the quadratic extension $F$.
Indeed, suppose that $F=F'$ where $F'=K[z]/((z^2-a'z-b')$.  Then the induced isomorphism $\phi: F[C_6]\rightarrow F'[C_6]$
respects the action of $C_2$, hence the fixed rings $H_0$, $H_0'$ are isomorphic as $K$-Hopf algebras.  
\end{remark}

\begin{example}  \label{H0} We assume that the base field $K=\Q$ and compute the structure of $H_0$ in the case that $L$ is the splitting field of $x^3-v$, irreducible over ${\mathbb Q}$.   In this case, $F=L^{\gen{r}}=\Q(\zeta_3)$, hence
$F=\Q[z]/(z^2+3)$, and so, $b=-3$, $u=-12$.    We then have $H_0=\Q[x,y]/I$, with 
\[I=(y^2+3x^2-12, (x-2)(x-1)(x+1)(x+2), (x-1)(x+1)(xy)).\]
The Hopf algebra structure of $H_0$ is given as
\[\Delta(\bar x)={1\over 2}\bar x\otimes \bar x-{1\over 6}\bar y\otimes \bar y,\]
\[\Delta(\bar y) =  {1\over 2}\bar x\otimes \bar y+{1\over 2}\bar y\otimes \bar x,\]
\[\varepsilon(\bar x)=2,\quad \varepsilon(\bar y)=0,\quad \sigma(\bar x)=\bar x,\quad \sigma(\bar y)=-\bar y.\]
\end{example}

We can also obtain the Wedderburn-Artin decomposition of $H_0$ in Example \ref{H0}.  

\begin{proposition}  Assume the conditions of Example \ref{H0}.  Then
\[H_0\cong \Q\times \Q\times \Q\times \Q\times \Q\times \Q,\] 
as $\Q$-algebras.  
\end{proposition}

\begin{proof}   The ideal $I$ determines an affine variety in $X\subseteq \Q^2$ consisting of exactly six points:
\[P_1=(-2,0), P_2=(-1,3), P_3=(1,3),P_4=(2,0), P_5=(1,-3), P_6=(-1,-3),\]
This is the set of common zeros of the polynomials in $I$ (Figure 1).

\vspace{.5cm}

\hspace{3.0cm} \includegraphics[height=3.0in]{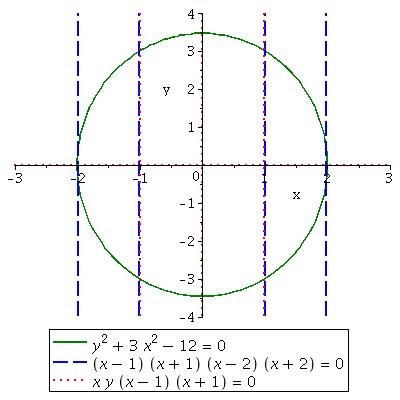}

\vspace{.2cm}

\hspace{3cm} Figure 1:\    Graph of variety determined by $I$. 

\vspace{.5cm}

As a commutative $\Q$-algebra, $H_0\cong \Q[x,y]/I$ is a product of fields with $\dim_{\Q}(H_0)=6$.  For $1\le j\le 6$, let 
$(x_j,y_j)$ be the coordinates of the point $P_j$ and let $\Psi_j: H_0\rightarrow \Q$ denote the ring homomorphism 
defined as $\bar x\mapsto x_j$, $\bar y\mapsto y_j$.   Then $\Psi_j$ is surjective and $\ker(\Psi_j)$ is an ideal $I_j$ of $H_0$ of dimension $5$ over $\Q$.  Now, $\bar y - y_j -6\bar x+6x_j\in I_j$, yet $\bar y - y_j -6\bar x+6x_j\not \in I_k$ whenever 
$j\not = k$.  Thus the ideals $I_j$, $1\le j\le 6$, are distinct and therefore must arise by omitting one factor isomorphic to 
$\Q$ from the Wedderburn-Artin decomposition of $H_0$.   It follows that the decomposition of $H_0$ must contain at least $6$ factors isomorphic to $\Q$, hence $H_0\cong \Q^6$. 
\end{proof}

\end{document}